\documentclass[12pt]{amsart}

\usepackage[utf8]{inputenc}

\usepackage{amssymb,latexsym}
\usepackage[curve]{xypic}

\newtheorem{theorem}{Theorem}[section]
\newtheorem{proposition}[theorem]{Proposition}
\newtheorem{corollary}[theorem]{Corollary}
\newtheorem{remark}[theorem]{Remark}
\newtheorem{lemma}[theorem]{Lemma}

\newtheorem{question}[theorem]{Question}

\input xy
\xyoption{all}

\numberwithin{equation}{section}

\usepackage[all]{xy}
\begin{document}


\vspace{2cm}






\title[Hilbert scheme]{On the Hilbert scheme of the moduli
space of vector bundles over an algebraic curve}

\author{L. Brambila-Paz}

\address{CIMAT, Apdo. Postal 402, C.P. 36240.
 Guanajuato, Gto,
M\'exico}

\email{lebp@cimat.mx}

\author{O. Mata-Gutierrez }

\address{CIMAT, Apdo. Postal 402, C.P. 36240.
 Guanajuato, Gto,
M\'exico}

\email{osbaldo@cimat.mx}



\date{}
\begin{abstract} Let $M(n,\xi)$ be the moduli space
of stable vector bundles of rank $n\geq 3$ and fixed determinant
$\xi$ over a complex smooth projective algebraic curve $X$ of
genus $g\geq 4.$ We use the gonality of the curve and $r$-Hecke
morphisms to describe a smooth open set of an irreducible
component of the Hilbert scheme of $M(n,\xi)$, and to compute its
dimension. We prove similar results  for the scheme of morphisms
$Mor_P(\mathbb{G},M(n,\xi))$ and  the moduli space of stable
bundles over $X\times \mathbb{G},$ where $\mathbb{G}$ is the
Grassmannian $\mathbb{G}(n-r,\mathbb{C}^n)$. Moreover, we give
sufficient conditions for $Mor_{2ns}(\mathbb{P}^1,M(n,\xi))$ to be
non-empty, when $s\geq 1$.

\end{abstract}

\maketitle

\section{Introduction}

Studying the geometry of varieties frequently involves the
understanding of their subvarieties. The Hilbert scheme $Hilb_Y$,
the Chow scheme $C(Y)$ and the scheme of morphisms $Mor(-,Y)$
 each provides a certain compactification of the space of
irreducible subvarieties of a projective variety $Y.$

 Let $M(n,\xi)$ be the moduli space
of stable vector bundles of rank $n\geq 3$ and fixed determinant
$\xi$ over a complex smooth projective algebraic curve $X$ of
genus $g\geq 4.$  In this paper, we are interested in studying the
Hilbert scheme $Hilb_{M(n,\xi)}^P$ for a fixed Hilbert polynomial
$P$ of degree $>1$.

It is worth pointing out that for  Hilbert polynomials $P$ of
degree $1$ it was proved in \cite{ram2} that
 there exists a component of $Hilb^{P}_{{M(2,\mathcal{O})}}$
 that provides a non-singular
model for $M(2,\mathcal{O})$.  For $M(2,\mathcal{O}_X(-1)),$ the
Hilbert scheme $Hilb^{m+1}_{{M(2,\mathcal{O}_X(-1))}}$ and the
Chow scheme of degree $1$ curves coincide with the moduli space of
stable maps $\bar{M}_0(M(2,\mathcal{O}_X(-1)),1)=Mor_1(\mathbb{P}^1,M(2,\mathcal{O}_X(-1)))
$ (see \cite{vicente} and
\cite{kilaru}). In \cite{kiem}, Kiem proved that the Hilbert
scheme $Hilb^{2m+1}_{{M(2,\mathcal{O}_X(-1))}}$ of conics, the scheme of morphisms
$Mor_2(\mathbb{P}^1,M(2,\mathcal{O}_X(-1)))$ and the
Chow scheme of conics are related by
contractions. To prove these results, they used the Hecke
correspondence and the space of Hecke curves. The {\it Hecke
correspondence} for vector bundles, was introduced by Narasimhan
and Ramanan  in \cite{ram1} and has proved to be one of the most
powerful tools in the study of $M(n,\xi)$. It can be summed up in the following diagram
(see \cite{ram1} and \cite{ram2})
$$
\xymatrix@1{
&{\mathbb{P}(U_x)}\ar[dl]_{\pi}\ar[dr]^{q}& \\
M(n,\xi (x))&&M(n,\xi ),}\\
$$
where $x\in X,$ $\pi:\mathbb{P}(U_x)\to M(n,\xi (x))$ is
the restriction of the projective Poincar\'e bundle to
$\{x\}\times M(n,\xi (x))$ and, at a generic point $F\in M(n,\xi ),$
$q:\mathbb{P}(U_x)\to M(n,\xi ) $ is a projective fibration with fibre $\mathbb{P}(F_x).$
 Narasimhan and Ramanan used
$(0,1)$ and $(1,0)$-stable bundles (see \S 2.3) to define
subschemes of $M(n,\xi )$, called {\it good Hecke cycles}.
The main concern in \cite{ram1} was the description of the deformations of the moduli space $M(n,\xi )$; moreover,
 the lines in the good Hecke cycles and their properties were implicitly described
there. Such lines were called later {\it Hecke curves} by Hwang (see
\cite{hwang1} and \cite{hwang2}) and they turn out to be the minimal rational curves in $M(n,\xi )$
(see \cite{hwang2}, \cite{muksun} and \cite{sun}). The moduli
space and the tangent spaces  of the minimal curves in $M(n,\xi)$
have been studied principally for rank $2$ in \cite{muksun},
\cite{hwang1}, \cite{hwang2}, \cite{isong}, \cite{sun} and
\cite{castel}. The interest in studying $Mor(\mathbb{P}^1,
M(n,\xi))$ also has its origin in attempts to compute the quantum
cohomology of the moduli space $M(n,\xi),$ which has recently
become an important topic of research (see e.g. \cite{vicente};
\cite{witten}; \cite{bertram}; \cite{bertram2} and \cite{kiem}).

 For an integer $0<r<n$, the problem that we
address in this article is to provide a description of the set of
maps from the
 Grassmannians
$\mathbb{G}=\mathbb{G}(n-r,\mathbb{C}^n)$ to $M(n,\xi )$.
 Our purpose is to describe a smooth open set which is both an
irreducible component
 of $Hilb^{P}_{M(n,\xi)}$ and of $Mor _P(\mathbb{G},M(n,\xi ))$,
 where $P$ is the Hilbert polynomial associated to the Grassmannian $\mathbb{G}$.

 In order to state our results we recall from \cite{mata1} that using $(k,\ell)$-stable bundles, one can
generalise the idea of the good Hecke cycles to obtain the notion
of {\it Hecke Grassmannians} (see \S 2) in $M(n,\xi)$.
 It has recently been noted by  O.
Mata-Guti\'errez that through a very general point $E\in
M(n,\xi)$, there exists Hecke Grassmannians passing through $E$.
The Hecke Grassmmannian is called an {\it $r$-Hecke cycle}, with
$0<r<n$, if it defines a closed subscheme. The good Hecke cycles
are precisely those where $r=1$.

 The set of  $r$-Hecke
cycles form an irreducible family (see \S 3). Let $\mathcal{HG}$
be the irreducible component of the Hilbert scheme $Hilb _{M(n,\xi
)}^{P}$ of $M(n,\xi)$ containing $r$-Hecke cycles. The main idea
behind the study of $\mathcal{HG}$ is to apply the
$(k,\ell)$-stability (see \S 2.3). For integers $k,\ell ,$ $1\leq
r\leq n-1$
 chosen such that
$$ 0\leq k(n-1)+\ell+r<(n-1)(g-1)$$
and
$$0\leq k+(\ell+r)(n-1)<(n-1)(g-1),$$
we construct in Section \S 3 {\it a fibration  $p:\mathcal{A}\to
X$}  where the
 fibre $p^{-1}$ is identified with the set of $(k,\ell)$-stable
vector bundles of rank $n$ and determinant $\xi(rx)$, for all $x\in X.$
Now we state our main result which is a generalisation of those in \cite{ram2}

\begin{theorem}\label{teoprin} If $(n,d)=1$ and $r$ is less than the gonality of $X$
then
\begin{enumerate}
\item there is an algebraic isomorphism $\Upsilon $ from
$\mathcal{A}$ to an open subscheme of $\mathcal{HG}$. \item The
Hilbert scheme $\mathcal{HG}$ is smooth at $\Upsilon (z),$ for any
$z \in \mathcal{A}.$ \item  $\dim \mathcal{HG}= (n^2-1)(g-1)+1.$
\item Locally, the deformations of $r$-Hecke cycles are $r$-Hecke
cycles.
\end{enumerate}
\end{theorem}

It would be
interesting to prove, in the non-coprime case, that our component
 is actually smooth and provides a non-singular
model for $M(n,\xi)$, as was proved for $n=2 $ and $\xi$ the
trivial bundle in \cite{ram1}.

The proof of Theorem \ref{teoprin} is similar in spirit to the proof of \cite[Theorem 5.13]{ram2}
(see also \cite{tyurin}).
 While some of the results in \cite{ram1} and
\cite{ram2} apply in our case, the main results need  extra
hypotheses, such as a condition on the gonality of the curve.
 We use the gonality of the curve to prove that
 any  $z\in \mathcal{A}$ defines a closed embedding $\phi _z:\mathbb{G}\to M(n,\xi)$ of a
 Grassmannian $ \mathbb{G}$ to $M(n,\xi)$,
 and for the injectivity of $\Upsilon:\mathcal{A}\to  \mathcal{HG} $.
 In our case we construct a diagram (see diagram
\ref{diaguni})
$$
\xymatrix@1{
&\mathfrak{G}\ar[dl]_{\pi _1}\ar[dr]^{\Phi}& \\
\mathcal{A}&&M(n,\xi)}\\
$$
with $\pi _1$ a  Grassmannian  fibrations and the fibre of $\Phi$
at $F$ the Grassmannian bundle $p:\mathbb{G}(r,F)\to X,$ for each
$F$ in a suitable open set $\mathcal{B} \subset M(n,\xi)$. This
construction generalises that of  \cite{tyurin}, \cite{ram1}, and
\cite{ram2}.

One of the advantages of using Grassmannians $\mathbb{G}$ lies in
the fact that they have natural subvarieties, namely the Schubert
varieties. This allows us to study flag Hilbert schemes
parameterising  $t-$tuples $(\mathbb{G}, Y_1, \dots , Y_t)$ such
that $ Y_1 \subset \dots \subset Y_t \subset \mathbb{G}  $ are
subschemes of $M(n,\xi).$ However, this topic exceeds the scope of
this paper. We just consider nested pairs of subschemes
$\mathbb{P}^1 \stackrel{i}{\hookrightarrow }
\mathbb{G}\stackrel{j}{\hookrightarrow } M(n,\xi)$. In particular,
our viewpoint sheds some new light on the study of rational curves
on $M(n,\xi)$ allowing us to obtain many more rational curves of
higher degrees.

\begin{corollary}\label{coroprin}
If $Mor_s(\mathbb{P}^1,\mathbb{G})\ne \emptyset$ then
$Mor_{2ns}(\mathbb{P}^1,M(n,\xi))\ne \emptyset. $
Moreover, $\dim Mor_{2ns}(\mathbb{P}^1,M(n,\xi)) \geq (n^2-1)(g-1)+1.$
\end{corollary}

The fibration $\mathcal{A}$ in Theorem \ref{teoprin} parameterises
closed subschemes, morphisms $f:\mathbb{G}\to M(n,\xi)$ and stable
vector bundles over $X\times \mathbb{G}$. To see this, recall that
the morphisms from $\mathbb{G}$ to $M(n,\xi)$ are parameterised by
a locally Noetherian scheme $Hom(\mathbb{G},M(n,\xi))$ and the
quotient
$$Mor_{P}(\mathbb{G},M(n,\xi))=
Hom_{P}(\mathbb{G},M(n,\xi))/Aut(\mathbb{G})$$ can be defined by
means of the Chow variety (see \cite{grot}, \cite{bertram} and
\cite{richard}).

We say that a morphism
 $\lambda :\mathbb{G} \to M(n,\xi)$
is  an {\it $r$-Hecke morphism}
if $\lambda (\mathbb{G}) $ is an $r$-Hecke cycle.
 Such morphisms have minimal degree $2n$
(see Proposition \ref{propdegree} and Corollary \ref{cor22}) and are in an
irreducible family parameterised by $\mathcal{A}$.
Let
$Mor^{\mathcal{H},r}_P(\mathbb{G},M(n,\xi))$ be the irreducible component
of the scheme  $Mor_{P}(\mathbb{G},M(n,\xi))$ that
contains the $r$-Hecke morphisms.

The next theorem is analogous to
Theorem \ref{teoprin}.

\begin{theorem}\label{teoprin2} Under the hypotheses of Theorem \ref{teoprin}
there is an algebraic injective morphism
 $$\Sigma:\mathcal{A}\to Mor^{\mathcal{H},r}_P(\mathbb{G},M(n,\xi)).$$
 Moreover, $Mor^{\mathcal{H},r}_P(\mathbb{G},M(n,\xi))$ is smooth at the $r$-Hecke morphisms and
 $$\dim Mor^{\mathcal{H},r}_P(\mathbb{G},M(n,\xi))= (n^2-1)(g-1)
 +1.$$
In particular, $\dim  Mor^{\mathcal{H},1}_P(\mathbb{P}^2, M(3,\xi)) = 8g-7.$
\end{theorem}

For any $z\in \mathcal{A}$, the  morphism $\Sigma (z)=\phi
_z:\mathbb{G}\to M(n,\xi)$ is defined by the existence of a vector
bundle $\mathcal{P}_z$ over $X\times \mathbb{G}$, which we
call an $r$-$Hecke$ $bundle$. The stability of $\mathcal{P}_z$,
with respect to any polarisation $L$, is established in
Proposition \ref{teoprin311}. Let $M^{L,\mathcal{H}}_{{X\times
\mathbb{G}}}(\mathcal{P}_z) $ be
 the irreducible
 component of {\it the moduli space of $L$-stable vector bundles over
 $X\times \mathbb{G}$ that contain $\mathcal{P}_z.$} Our last
result  describes a smooth open set
of $M^{L,\mathcal{H}}_{{X\times \mathbb{G}}}(\mathcal{P}_z) .$ Furthermore,
 we see that locally the deformations
of $\mathcal{P}_z$ come from those of the curve, of the
Grassmannian as well as those induced by the elements of
$H^0(\mathbb{G},\phi_z^*(TM) )$.

\begin{theorem}\label{teoprin3} Under the hypotheses of Theorem \ref{teoprin}
\begin{enumerate}
\item there is an algebraic injective morphism
 $\Gamma :\mathcal{A}\to M^{L,\mathcal{H}}_{{X\times \mathbb{G}}}(\mathcal{P}_z).$
  \item $M^{L,\mathcal{H}}_{{X\times \mathbb{G}}}(\mathcal{P}_z)$
 is smooth at $\Gamma (z)$, for all $z\in \mathcal{A}.$
 \item  $\dim M^{L,\mathcal{H}}_{{X\times \mathbb{G}}}(\mathcal{P}_z)= n^2g+1.$
 \item $\dim M^{L,\mathcal{H}}_{{X\times
\mathbb{G}}}(\mathcal{P}_z)/(Aut(X)\times Aut(\mathbb{G}))
=(n^2-1)(g-1) +1 .$
\end{enumerate}
\end{theorem}

This article is organised as follows. In Section \S $2$, we
summarise the relevant material on Grassmannians,
$(k,\ell)$-stability  and elementary transformations. Some of the
recent results are reviewed in a more general setting. Section \S
$3$ is devoted to the study of
 the morphisms  $\phi _z$ and  $\Upsilon $. The main result, Theorem \ref{teoprin}, is proved in the fourth section.
In Section \S 5, we touch only a few aspects of the theory of
rational curves on $ M(n,\xi)$ and prove Theorem \ref{teoprin2}
and \ref{teoprin3}. We raise some  questions in Section \S $6$.

\bigskip

{\it Notation:} Given a vector bundle $E$ over $X$ we denote by
$d_E$ the degree, by $n_E$ the rank and by $\mu
(E):=\frac{d_E}{n_E}$ the slope of $E.$  For abbreviation, we
write $H^i(E)$ instead of $H^i(X,E),$ whenever it is convenient.
 The Grassmannian of $s$-planes
 of  a vector space or of a vector bundle $V$ will be denoted by $\mathbb{G}(s,V).$
 By $p_i$, we mean the natural projection to the $i$-factor.  The
 trivial bundle over $Y$ with fibre the vector space $W$ will be denoted by  $\mathcal{O}_Y\times W.$

\section{Grassmannians, $(k,\ell)$-stability and elementary transformations}

Let $M(n,d)$ be the moduli space of stable vector bundles of rank
$n\geq 3$ and degree $d$ over $X.$ The fibre of the
determinant morphism $det:M(n,d)\to Pic^d(X)$ at $\xi \in Pic^d(X) $ is
 $M(n,\xi)$.
 It is well known that $M(n,\xi)$ is a Fano variety
with Picard group $\mathbb{Z}\Theta$, where $\Theta$ is an ample
divisor. If $K_M$ is the canonical bundle then $-K_M=2(n,d)\Theta
$ (see \cite{dn}). When $n$ and $d$ are coprime $M(n,\xi)$ is
projective, smooth of dimension $(n^2-1)(g-1)$ and there is a
Poincar\'{e} bundle $\mathcal{U}$ and a Grassmannian Poincar\'{e}
bundle $\mathbb{G}(s,\mathcal{U})$. If $n$ and $d$ have a common
divisor, there is no universal vector bundle (see \cite{rama};
also \cite{news}). In \cite[Proposition 2.4]{ram1} it was proved
that in the non-coprime case there exist an  \'{e}tale cover of
$M(n,\xi) $ and a family $\mathcal{V}$ of stable bundles of rank
$n$ and determinant $\xi$ with universal properties.
 However, as in the projective case (see \cite{bbn}),
 there always exists a Grassmannian Poincar\'{e} bundle
$$\mathbb{G}(s,\mathcal{U})\to X \times M(n,\xi)$$ with
the property that its restriction to
$X\times\{E\}$ is isomorphic to the Grassmannian bundle $\mathbb{G}(s,E)$ over $X$
that parameterises all the  $s$-planes in the fibres
of the stable vector bundle $E$ on $X.$

Since Grassmannians, elementary transformations and $(k,\ell)$-stability
will be the main tools that we will use, we briefly review the essentials, and fix the notation.

\subsection{Grassmannians}

 We will give only the principal properties of Grassmannians of
vector spaces and of vector bundles that we use. For a fuller
treatment we refer the reader to, e.g., \cite{tyurin}.

Let $E$ be a vector bundle of rank $n$ over $X.$
Let $p_E:\mathbb{G}(n-r, E)\to X$ be the Grassmmannian bundle whose
fibre at $x\in X$ is the Grassmannian $\mathbb{G}(n-r, E_x)$
 of $(n-r)$-planes of $E_x$. It is clear that
$\mathbb{G}(n-r,{{E}})= \mathbb{G}(r,{{E}}^{{*}} ).$

Let
\begin{equation}\label{eqgrass2}
\xi _E: 0\to\mathcal{S}_E\to p_E^*E \stackrel{{\alpha _E}}{\to}
Q_E\to 0
\end{equation}
 be the tautological exact sequence over $\mathbb{G}(n-r,{{E}}),$ where $S _{{E}}$  and
$Q_{{E}}$ are the tautological bundles of rank $n-r$ and $r$,
respectively. The tangent bundle of
$\mathbb{G}(n-r,E)$, denoted by $T\mathbb{G}(E)$, fits in the
following  extension
\begin{equation}\label{eqtan}
\zeta: 0\to T_{{p_E}}\to T\mathbb{G}(E) \to p^*_ETX\to 0 ,
 \end{equation}
 where $T_{{p_E}}$ is the tangent bundle to the fibres and $T_{{p_E}}=S ^*_E\otimes Q_E.$

\smallskip

For any $x\in X$, let
\begin{equation}\label{eqgrass1}
\xi _{{E_x}}: 0\to\mathcal{S}_{{E_x}}\to \mathcal{O}_{\mathbb{G}}\times {{E_x}}
\stackrel{\alpha _{{E_x}}}{\longrightarrow} Q_{{E_x}}\to 0
\end{equation}
 be the restriction of the sequence $(\ref{eqgrass2})$ to $\mathbb{G}(n-r,{{E_x}}).$

The Grassmannian varieties $\mathbb{G}=\mathbb{G}(n-r,{{E_x}})$
are  Fano varieties with Picard group
$\mathbb{Z}\mathcal{O}_{\mathbb{G}}(1)$, where
$\mathcal{O}_{\mathbb{G}}(1)$ is an ample divisor. The tangent
bundle $T\mathbb{G}$ of $\mathbb{G}$ is the vector bundle $ S
^{{*}}_{{E_x}}\otimes Q_{{E_x}}.$ Since det$(Q_{{E_x}})=
\mathcal{O}_{\mathbb{G}}(1),$ $T\mathbb{G}$ has degree $n$. Denote
by $\Omega ^1 \mathbb{G}$ the cotangent bundle of $\mathbb{G}.$

The following remark summarises the main properties of the
cohomology of $S _{{E_x}}$, $T\mathbb{G}$ and $\Omega ^1
\mathbb{G}$, which will be needed in Section $4$ (see for instance
\cite{tyurin}).

\begin{remark}\begin{em}\label{emgras0} Let $\mathbb{G}$
be the Grassmannian $\mathbb{G}(n-r,{{E_x}})$
\begin{enumerate}
\item  $H^i(\mathbb{G}, S _{{E_x}})=0$ for all $i\geq 0$; \item
$\dim H^0(\mathbb{G}, T\mathbb{G})=n^2-1$
 and $ H^i(\mathbb{G}, T\mathbb{G})=0$ for $i\geq 1$;
\item $H^1( \mathbb{G},\Omega ^1 \mathbb{G})=\mathbb{C} \ \
\mbox{and} \ \ H^i( \mathbb{G},\Omega ^1 \mathbb{G})=0, \
\mbox{for} \ i\ne 1. $
\end{enumerate}
 \end{em}\end{remark}

\bigskip

Given a morphism $\lambda :\mathbb{ G}\to M(n,\xi)$ we define the degree
of $\lambda $  as $$d(\lambda):=degree(\lambda ^*(-K_M)).$$

From \cite[4(c)]{grot}, the morphisms from $\mathbb{G}$ to
$M(n,\xi)$ are parameterised by a locally Noetherian scheme
$Hom(\mathbb{G},M(n,\xi))$. We can see $Hom(\mathbb{G},M(n,\xi))$
as a subscheme of $Hilb_{(\mathbb{G} \times M(n,\xi))}$ when we
identify a morphism $\lambda$ with its graph $\Gamma _\lambda$ in
$ \mathbb{G}\times M(n,\xi).$ As it is well known,
$Hom(\mathbb{G},M(n,\xi))$ is the disjoint union of the subschemes
$Hom_P(\mathbb{G},M(n,\xi))$, for all polynomials $P$, where
$Hom_P(\mathbb{G},M(n,\xi))$ is the
 subscheme that
parameterises morphisms $\lambda:\mathbb{G}\to M(n,\xi)$ with
fixed Hilbert polynomial $P.$
 Denote by $Mor_P(\mathbb{G},M(n,\xi))$ the scheme
 $$Hom_P(\mathbb{G},M(n,\xi))/Aut(\mathbb{G}),$$
where $\lambda  \sim \lambda '$ if there exists $\beta \in
Aut(\mathbb{G})$ such that $\lambda '= \lambda \circ \beta .$ The
scheme $Mor_P(\mathbb{G},M(n,\xi))$ can be defined by means of the
Chow variety (see \cite{grot}, \cite{bertram} and \cite{richard}).

\begin{remark}\begin{em}\label{remtanggras} From \cite{grot} and \cite{kollar} we have that
\begin{enumerate}
\item  the expected dimension of  $Hom_P(\mathbb{G},M(n,\xi))$ is
$h^0(\mathbb{G},\lambda ^*TM(n,\xi)).$ \item
$Hom_P(\mathbb{G},M(n,\xi))$ is smooth at $\lambda$ if
$h^1(\mathbb{G},\lambda ^*TM(n,\xi))=0$. \item  $H^0(\mathbb{G},
T\mathbb{G})$ is the tangent space at the identity to the group of
automorphisms of $\mathbb{G}$, and the image of the canonical
morphism $$H^0(\mathbb{G}, T\mathbb{G}) \to H^0(\mathbb{G},
\lambda ^*(TM))$$ corresponds to the deformations of $\lambda$ by
reparameterisations. \item The expected dimension of
$Mor_P(\mathbb{G},M(n,\xi))$ is the dimension of $$H^0(\mathbb{G},
\lambda ^*(TM))/H^0(\mathbb{G}, T\mathbb{G}).$$ \item
$Mor_P(\mathbb{G},M(n,\xi))$ is smooth at $\lambda$ if
$h^1(\mathbb{G},\lambda ^*TM(n,\xi))=0$.
\end{enumerate}\end{em}\end{remark}

\smallskip

To define morphisms $\lambda : \mathbb{G}\to M(n,\xi)$ we use
the {\it r-elementary transformations} and {\it $(k,\ell)$-stability}.

\subsection{$r$-Elementary transformations}

Let $E$ be a vector bundle over $X$ of rank $n$ and determinant
$\eta$ of degree $e$. For any  $\wp=(x, W\subset E_x)\in
\mathbb{G}(n-r,E)$ consider a new vector bundle $E^W$  defined by
the exact sequence of sheaves
\begin{equation}\label{eq0}
\xi _{x,W}: 0\to E^W\stackrel{\iota}{\to} E\stackrel{\alpha
_W}{\longrightarrow} \mathcal{O}_x\times (E_x/W)\to 0,
\end{equation}
where $\mathcal{O}_x\times (E_x/W)$ is the skyscraper sheaf with
support in $x$ and fibre $E_x/W$. That is, $E^W= Ker(\alpha _W)$.
 The vector bundle $E^W$ is called the {\it $r$-elementary transformation
of $E$ in $\wp=(x,W\subset E_x)$}. Note that $E^W$ has rank $n,$
degree $e-r$ and determinant $\eta \otimes \mathcal{O}(-rx).$

\begin{remark}\begin{em}\label{remkernel}Let $E$ be a vector bundle over $X$.
\begin{enumerate}
\item Denote by $Ker(\iota _x)$ the kernel of the homomorphism
$E_x^W\stackrel{\iota _x}{\to} E_x$ between the fibres at $x,$
induced by the sheaf map $\iota $ in $(\ref{eq0})$. It has
dimension $r$ and its annihilator $Ker(\iota _x)^\bot$ is a $n-r$
dimensional subspace in $(E^W_x)^*$, which is canonically
isomorphic to $W^*.$ \item The restriction of $(\ref{eq0})$ to $x$
gives the exact sequence
\begin{equation}\label{eqkernel00}
0\to Ker(\iota _x)\to E^W_x\stackrel{\iota _x}{\to}
E_x\stackrel{\alpha _{{W_x}}}{\to} (E_x/W)\to 0,
\end{equation}
 which splits as
\begin{equation}\label{eqiota1}
0\to Ker(\iota _x)\to E^W_x\stackrel{\iota _x}{\to} W\to 0 \ \
\mbox{and} \ \ 0\to W\to  E_x\stackrel{\alpha _{{W_x}}}{\to}
(E_x/W)\to 0.
\end{equation}
(The extension $(\ref{eqkernel00})$ and its splitting
$(\ref{eqiota1})$ will be relevant in \S 4, (see Proposition
\ref{propprin})).
 \item A point $\wp =(x, W\subset E_x)$ in $\mathbb{G}(n-r,E)$
 defines the element $$\tilde{\wp }:=(x, Ker(\iota _x)\subset E^W_x)$$ in $\mathbb{G}(r,E^W).$
 \item The $r$-elementary transformation of $(E^W)^{{*}}$ in $(x, W^*\subset (E^W)_x^{{*}} )$ is $E^{{*}}$.
 So, $$((E^W)^*)^{{W^*}}=E^*,$$ and we have the exact sequence
\begin{equation}\label{eq1}
\xi_{x,W^*}: 0\to E^{{*}}\stackrel{\iota}{\to} (E^W)^{{*}}\to
\mathcal{O}_x\times ((E^W)^*_x/W^*)\to 0.
\end{equation}
\end{enumerate}
\end{em}
\end{remark}

\smallskip

We are interested in describing the set
$$\Omega := \{ (x, E \stackrel{\alpha}{\twoheadrightarrow} \mathcal{O}_x^r)
 : E \ \ \mbox{and} \ \ Ker(\alpha) \ \ \mbox{are stable }  \}.$$
 That is, in describing the set of $r$-elementary transformations of stable bundles that give stable bundles.
 To describe $\Omega $ we use the $(k,\ell)$-stability.

\subsection{$(k,\ell)$-stability}

Let $k$ and $\ell$ be integers. A vector bundle $E$ over $X$ is
$(k,\ell)$-stable (see \cite{ram2}) if for any proper subbundle
$F$ of $E$
$$ \frac{k(n-n_F)+\ell n_F}{nn_F}< \mu(E)-\mu(F).$$
In particular, if $k$ and $\ell $ are non-negative integers,
$(k,\ell)$-stability implies stability. However, for negative
values of $k$ and $\ell$, a $(k,\ell)$-stable bundle does not need to be
stable (see \cite{mata1}).

 Denote by
$A_{k,\ell}(n,d)$ the set of $(k,\ell)$-stable vector bundles of
rank $n$ and degree $d$ over $X$ and let $A_{k,\ell}(n,\xi):= A_{k,\ell}(n,d) \bigcap M(n,\xi).$
 Proposition 5.3 in \cite{ram2} shows that $(k,\ell)$-stability is an open condition.
 In particular,  $(k,\ell)$-stable
bundles are very general in $M(n,\xi )$. A point is very general
if it is outside of a countable union of subvarieties of dimension
strictly smaller than the dimension of $M(n,\xi)$.

There are natural filtrations (see \cite{mata1})
$$A_{k,\ell}(n,d) \supset A_{k,\ell+1}(n,d)\supset A_{k,\ell+2}(n,d)...$$ and
$$A_{k,\ell}(n,d) \supset A_{k+1,\ell}(n,d)\supset A_{k+2,\ell}(n,d)...$$
with $A_{0,0}(n,d)=M(n,d)$ and $A_{{k_0,\ell_0}}(n,d)=\emptyset$
if
$$ k_0(n-1)+\ell _0\geq (n-1)g \ \ \mbox{or} \ \ k _0+\ell_0(n-1)\geq (n-1)g.$$

The non-emptiness of $A_{k,\ell}(n,\xi)$  is established in
 \cite[Proposition 2.4]{mata1}. The following
proposition is a reformulation of \cite[Proposition 2.4]{mata1}
and \cite[Lemma 5.5]{ram2} (see also \cite{mata1}) in terms of
$r$-elementary transformations.

\begin{proposition}\label{propkl} If $k,\ell,r,n,d $ are integers such that
 \begin{equation}\label{eqprinkl}
 0< k(n-1)+\ell+r<(n-1)(g-1)\\
\mbox{ and } \\
0< k+(\ell +r)(n-1)<(n-1)(g-1)
\end{equation}
then $A_{k,\ell}(n,d)\ne \emptyset ,$ $A_{k,\ell}(n,d)\subset
M(n,d)$ and $A_{k,\ell}(n,\xi)$ is non-empty. Moreover,  any
$r$-elementary transformation of a $(k,\ell)$-stable bundle in
$A_{k,\ell}(n,\xi )$ is $(k,\ell-r)$-stable; in particular, it is stable
and general.
\end{proposition}

The principal significance of the above proposition is that it allows
us to define maps from  Grassmannians to $M(n,d)$.
 Indeed, for integers $k,\ell,r,n,d $ as in Proposition
 \ref{propkl},
 and  $z=(x,E)\in X\times A_{k,\ell}(n,d),$ we have a
map $$\phi _z :\mathbb{G}(n-r,E_x) \to  M(n,d-r),$$ defined as
$W\mapsto E^W.$ Our aim is to define morphisms from
Grassmannians to $M(n,\xi)$, for a fix $\xi \in Pic^{d}(X)$.

\begin{remark}\begin{em}\label{rem1r} The map $\phi _z :\mathbb{G}(n-r,E_x) \to  M(n,d-r)$
can be defined by considering just $(0,r)$-stable bundles.
However, for our purpose (see \S 5) it is convenient to use
$(k,\ell +r)$-stable bundles that satisfy $(\ref{eqprinkl})$,
since in this case the vector bundles in the image $\phi _E
(\mathbb{G}(n-r,E))$ will be general.
\end{em}\end{remark}

\section{The morphisms $\phi _z$ and $\Upsilon $ }

We assume throughout the rest of the article, unless otherwise
stated, that ${k,\ell}$ and $r$ are integers satisfying the
inequalities in Proposition \ref{propkl}. For the construction of the
morphisms we also assume that $(n,d)=1$, and hence $(n,d+r)=1.$
However, we can always work in an \'{e}tale cover (see
\cite[Proposition 2.4]{ram1}).

To construct algebraic morphisms from Grassmannians to $M(n,\xi)$
 we first define  a fibration $\mathcal{A}\to X, $ which depends on
$k,\ell, r,n $ and $\xi$, but to streamline the notation, we omit
such indexes.

Fix $\xi \in Pic^d(X).$ Define
$$\vartheta :X\times A_{k,\ell}(n,d+r) \to Pic ^d(X)\ \ \ \mbox{as} \ \
(x,E)\mapsto  \mathcal{O}_X(-rx)\otimes det(E).$$   Let
$\mathcal{A} $ be the inverse image $\vartheta ^{-1}(\xi).$ The
natural map $\pi : \mathcal{A} \to X$ is a fibration with fibre
$A_{k,\ell}(n,\xi (rx))$ at $x\in X$. Note that the $r$-elementary
transformation associated to the elements in $\mathcal{A} $ have
determinant $\xi$. For any $z=(x,E)\in \mathcal{A}$ we  denote
$\mathbb{G}(n-r,E_x)$ as $\mathbb{G}(z).$

To prove that given $z=(x,E)\in \mathcal{A}$ the map
$$
\phi _z:\mathbb{G}(z)\to  M(n,\xi) \ \ \ \  W  \mapsto  E^W
$$
 is algebraic, we construct a family
of stable bundles parameterised by $\mathbb{G}(z)$.
The construction of the family is similar to that
 in \cite{tyurin}; \cite{ram1}; and \cite{ram2}.
 Since it is relevant for our work, we recall the main details.

The construction goes as follows. Given $z=(x,E),$ denote by
${\alpha }_z$
 the surjective homomorphism ${\alpha }_z :p_1^*E \to p_1^*\mathcal{O}_x\otimes
 p_2^*Q_{{E_x}}$ of sheaves over $X\times \mathbb{G}(z)$ associated to the surjective morphism
 $\alpha _{{E_x}}:\mathcal{O}_G\times E_x \to Q_{{E_x}}$ under the isomorphism
 $$H^0(\mathbb{G}(z),Hom(
\mathcal{O}_G\times E_x, Q_{{E_x}})) \cong H^0(X\times
\mathbb{G}(z),Hom(p_1^*E, p_1^*\mathcal{O}_x\otimes
p_2^*Q_{{E_x}})),$$ where $p_i$ is the projection, for $i=1,2.$

Since $p_1^*\mathcal{O}_x\otimes p_2^*Q_{{E_x}}$ has a locally free
resolution of length $1$, the kernel of
$\alpha _z :p_1^*E^{{*}} \to p_1^*\mathcal{O}_x\otimes  p_2^*Q_{{E_x}},$
denoted by ${\mathcal{P}_z},$ is locally free and fits in the exact sequence
\begin{equation}\label{eq2}
\xi _{x,E}:  0\to {\mathcal{P}_z}{\to} p_1^*E\stackrel{\alpha _z }{\to} p_1^*\mathcal{O}_x\otimes
 p_2^*Q_{{E_x}} \to 0
 \end{equation}
of sheaves over $X\times \mathbb{G}(z)$.  The restriction of
(\ref{eq2}) to $X\times \{W\}$ for any $W\in \mathbb{G}(z)$ is
precisely the extension (\ref{eq0}), i.e.
${\mathcal{P}_z}_{{|_{X\times \{W\}}}}= E^W$. Therefore, if
$z\in \mathcal{A}$,
\begin{equation}\label{eqbundle}
{\mathcal{P}_z}\to X\times \mathbb{G}(z)
\end{equation}
is a family of stable bundles of rank $n$ and determinant
$\xi$ parameterised by $\mathbb{G}(z)$ and hence we have a morphism $$\phi
_z: \mathbb{G}(z) \to M(n,\xi )$$  with the following properties.

\begin{proposition}\label{prop1} Under the hypotheses of Proposition \ref{propkl}, if $r$ is less than
the gonality of $X$ then, for $z\in \mathcal{A}$,
  the morphism $\phi
_z: \mathbb{G}(z) \to M(n,\xi )$  is a
closed embedding for any $z\in \mathcal{A}$. Moreover, $\phi
_z(\mathbb{G}(z)) \subset A_{k,l-r}(n,\xi).$
\end{proposition}

\begin{proof} First observe
 that since $r$ is less than the gonality of $X$,
$h^0(\mathcal{O}(rx))\leq 1.$ Thus, any morphism $\xi \to \xi
\otimes \mathcal{O}(rx)$ vanishes only at $x$.

Let us prove the injectivity of $\phi_z$. Suppose, contrary to our
claim, that there exist $W,V\in \mathbb{G}(z)$
  such that $W\ne V$ but $F:=E^W=E^V.$ Hence, $\bigwedge
^nF= \xi $ where $\xi = \bigwedge ^n E \otimes
\mathcal{O}_X(-rx)$. From $(\ref{eq0})$, we have two non-zero
linearly independent homomorphisms $f_1,f_2:F\to E$, that depend
on $W$ and $V$ respectively. Choose $y\ne x \in X$ and a
linear combination $\lambda _1 (f_1)_y +\lambda _2(f_2)_y$ that is
not an isomorphism. Then the homomorphism $g:=\lambda _1 (f_1)
+\lambda _2(f_2):F\to E$
 has no maximal rank at $y$. Therefore, the induced map $\bigwedge ^n g: \bigwedge ^n F \to
\bigwedge ^n E$ defines a non-zero section in
$H^0(\mathcal{O}(rx))$
 that vanishes on $y$ and $x,$ contrary to the gonality of
$X$.
 Therefore, $\phi _z$ is injective, which is our
claim.

The proof above gives more, namely
\begin{equation}\label{eqsimple}h^0(X,
Hom(E^W,E))=1,
\end{equation}
 for any $W\in \mathbb{G}(z)$.

To prove the injectivity of the differential map
$$d\phi _z :T_{[W]}\mathbb{G}(z)\to T_{{E^W}}M(n,\xi)$$
recall that $$ T_{[W]}\mathbb{G}(z)={W^*\otimes E_x/W}\subset
(E^W)_x^{{*}} \otimes {E_x/W} \ \  \mbox{and} \ \
T_{{E^W}}M(n,\xi)=H^1(X,ad(E^W)).$$

Tensor
the exact sequence (\ref{eq0}) with $(E^W)^{{*}}$ and
apply cohomology to get the exact sequence
\begin{equation}\label{eq3}
 0\to H^0(End( E^W) ) {\to} H^0((E^W)^{{*}}\otimes E ){\to} (E^W)_x^{{*}} \otimes {E_x/W}
  \stackrel{\delta }{\to} H^1(End( E^W))\to \cdots
 \end{equation}
Since $E^W$ is stable
 and $  H^0(X,Hom(E^W,E))\cong
\mathbb{C}$ (see $(\ref{eqsimple}))$, $(\ref{eq3})$ shows that the
coboundary map
 $(E^W)_x^{{*}} \otimes {E_x/W} \stackrel{\delta }{\to}
H^1(X,End(E^W))$
 is injective.

 Hence, the restriction of $\delta $ (or $(-\delta)$) to ${W^*\otimes E_x/W}$ is injective and is
 precisely  the differential  $d \phi _z$ (see  \cite[Lemma 5.10]{ram2}).
  Therefore $\phi _z : \mathbb{G}(z) \to M(n,\xi)$ is a closed
embedding, which proves the proposition.

\end{proof}

Given a pair $z\in \mathcal{A}$, the image $\phi
_z(\mathbb{G}(z))\subset M(n,\xi)$ defines a closed subscheme in
$M(n,\xi )$, called {\it $r$-Hecke cycle,} and hence a point
$[\phi _z(\mathbb{G}(z))]$ in the Hilbert scheme
$Hilb_{M(n,\xi)}$. We will identify $\mathbb{G} (z)$ with its
image $\phi_z(\mathbb{G} (z))$ in $M(n,\xi)$ when no confusion can
arise. The morphism $\phi_z$ is called {\it $r$-Hecke morphism}
and the vector bundle ${\mathcal{P}_z}\to X\times \mathbb{G}(z)$
an {\it $r$-Hecke bundle}. In Section \S 5 we will
parametrise such morphisms and bundles. \\

The construction of the morphism $\phi _z$ can be done for
families of $(k,\ell)$-stable bundles. Indeed, let $f:\mathcal{E}
\to X\times T$ be a family of $(k,\ell)$-stable bundles of rank $n$ and degree $d+r$
 parameterised by a scheme $T$. Define $\vartheta :X\times T \to Pic ^d(X)\ \ \mbox{as} \ \
(x,t)\mapsto  \mathcal{O}_X(-rx)\otimes det(\mathcal{E}_t)$ and
let $\mathcal{A}(T) $ be the inverse image $\vartheta ^{-1}(\xi).$
Let $\mathbb{G}(\mathcal{E})\stackrel{{\pi _1}}{\to}
\mathcal{A}(T)$ be the Grassmannian bundle of $(n-r)$-planes
associated to the restriction of $\mathcal{E}$ to
$\mathcal{A}(T)$. Let $\gamma :\mathbb{G} (\mathcal{E})\to X $ be
the composition $$ \mathbb{G}(\mathcal{E})\stackrel{{\pi _1}}{\to}
\mathcal{A}(T) \stackrel {p_1}{\to} X$$ and $\Gamma:=\Gamma
_{\gamma} \subset X\times \mathbb{G}(\mathcal{E})$ the divisor
associated to the graph of $\gamma $ (strictly speaking, $\Gamma
_{\gamma}$ is in $ \mathbb{G}(\mathcal{E})\times X$). As before,
over $X\times \mathbb{G}(\mathcal{E})$ we have a surjection
$$ p_2^*{\pi _1}^*(\mathcal{E})
\stackrel{\widetilde{\alpha} }{\to}\mathcal{O}_{\Gamma}\otimes
p_2^* Q_{\mathcal{E}} \to 0,$$
 where $Q_{\mathcal{E}}$ is the tautological quotient bundle and
 $p_2:X\times \mathbb{G}(\mathcal{E})\to \mathbb{G}(\mathcal{E})$ the projection.
The kernel $\mathcal{P}_{\mathcal{E}}$, which fits in the exact
sequence
\begin{equation}\label{equni1}
0\to \mathcal{P}_{\mathcal{E}} \to p_2^*\pi_1^*(\mathcal{E})
\stackrel{{\tilde{\alpha}}}{\to} \mathcal{O}_{\Gamma}\otimes p_2^*
Q_{\mathcal{E}} \to 0,
\end{equation}
is locally free. Hence, $\mathcal{P}_{\mathcal{E}}$ is a family
of stable bundles parameterised by $\mathbb{G}(\mathcal{E}).$
Therefore, we have a morphism
\begin{equation}\label{eqgrass}
\Phi _{\mathcal{E}}:\mathbb{G} (\mathcal{E}) \to M(n,\xi)
\end{equation}
and a diagram

\begin{equation}\label{diagT}
\xymatrix@1{
&\mathbb{G}(\mathcal{E})\ar[dl]_{\pi_{1}}\ar[dr]^{\Phi_{\mathcal{E}}}& \\
\mathcal{A}(T)&&M(n,\xi).}\\
\end{equation}

In particular, applying the above  construction to the family
defined by the restriction of the Poincar\'{e} bundle
$\mathcal{U}$ to $X\times \mathcal{A}_{k,\ell}(n,d+r)$, we obtain
the exact sequence (see $(\ref{equni1})$)
\begin{equation}\label{eqpoincare}
0\to \mathcal{P}_{\mathcal{U}} \to p_2^*\pi _1^*(\mathcal{U}) \to
\mathcal{O}_{\Gamma}\otimes p_2^* Q_{\mathcal{U}} \to 0,
\end{equation}
over $X\times \mathbb{G}(\mathcal{U})$ with $\pi _1:\mathbb{G}(
\mathcal{U})\to \mathcal{A}$; and the diagram
\begin{equation}\label{diaguni}
\xymatrix@1{
&\mathbb{G}(\mathcal{U})\ar[dl]_{\pi_{1}}\ar[dr]^{\Phi}& \\
\mathcal{A}&&M(n,\xi).}\\
\end{equation}
where $\Phi$ is $\Phi _{\mathcal{U}}.$

The above construction is functorial. If we denote by
$\mathcal{HG}$ the irreducible component of $Hilb_{M(n,\xi)}$ that
contains the $r$-Hecke cycles $[\phi_z(\mathbb{G} (z))],$  we get
an algebraic morphism
\begin{equation}\label{eqphi}
\Upsilon  : \mathcal{A} \to \mathcal{HG},
\end{equation} defined as $z\mapsto [\phi_z(\mathbb{G} (z))]$.

\begin{proposition}\label{propdist} If $(n,d)=1$ and $r$ is less than the gonality  of $X$
then  $\Upsilon  : \mathcal{A} \to
\mathcal{HG}$ is injective.
\end{proposition}

\begin{proof} Suppose the proposition were false. Then we could find
 $(x,E)=z_1\ne  z_2=(y,F)$ in $\mathcal{A}$
such that $[\phi_{{z_1}}(\mathbb{G} (z_1))]=
[\phi_{{z_2}}(\mathbb{G} (z_2))].$ Since $\phi _{{z_1}}$  and
$\phi _{{z_2}}$ are embeddings, we get an isomorphism $\beta :
\mathbb{G}(z_1)\to \mathbb{G}(z_2)$ that induces the
following commutative diagrams
$$
\xymatrix@1{\mathbb{G}(z_1)\ar[dd]_{\beta}\ar[rrd]^{\phi _{{z_1}}} & \\
&&M(n,\xi)\\
     \mathbb{G} (z_2)\ar[rru]^{\phi _{{z_2}}} & }\\
$$
and
$$
\xymatrix@1{X \times \mathbb{G}(z_1)\ar[rd]_{p_{1}}\ar[rr]^{(id,\beta)} &
& X \times \mathbb{G} (z_2)\ar[ld]^{\tilde{p_{1}}}\\
     &X& }\\
$$
i.e. $\phi _{{z_1}} = \phi _{{z_2}}\circ \beta $ and $\tilde{p} _1\circ (id,\beta) = p_1.$

By the universal properties of $M(n,\xi), $
\begin{equation}\label{eqiso}
{\mathcal{P}_{{z_1}}} \cong (id, \beta)^*({\mathcal{P}_{{z_2}}})\otimes
p_2^*(L)
\end{equation}
 where $L$ is a line bundle on $\mathbb{G}(z_1).$
But $L$ is trivial since $({\mathcal{P}_{{z_1}}})_{{|_{\{t\}\times
\mathbb{G} (z_1)}}} $ is trivial, for any $t\ne x$ (see
(\ref{eq2})). We thus get ${\mathcal{P}_{{z_1}}} \cong (id, \beta
)^*({\mathcal{P}_{{z_2}}})$ and $p_{1*}{\mathcal{P}_{{z_1}}} =
\tilde{p} _{1*}{\mathcal{P}_{{z_2}}}.$

Let
$$0\to
p_{1*}{\mathcal{P}_{{z_1}}}{\to}p_{1*}(p_1^* E)\stackrel{\alpha _*}{\longrightarrow}
p_{1*}(p_1^*\mathcal{O}_x\otimes
 p_2^*Q_{{E_x}}) \to \dots
 $$
 be the direct image sequence of $(\ref{eq2})$ by $p_1.$
 It follows that $p_{1*}({\mathcal{P}_{{z_1}}})\cong E(-x) $, because

\begin{enumerate}
\item $ p_{1*}p_1^*(E)\cong E,$
\item  $\ p_{1*}(p_1^*\mathcal{O}_x\otimes
 p_2^*Q_{{E_x}} )\cong  \mathcal{O}_x\otimes E$,
 \item the morphism $\alpha _* $ is
 the morphism associated to $\alpha _{{E_x}},$
 \item the kernel
of $E \stackrel{\alpha _*}{\to}  \mathcal{O}_x\otimes E$ is
$E(-x)$.
\end{enumerate}
 Similarly, $\tilde{p}_{1*}({\mathcal{P}_{{z_2}}})\cong F(-y)$. Therefore, since $r$ is less
 than the gonality of $X$, the isomorphisms
 $$
 \begin{array}{ccl}
 E(-x)&\cong &p_{1*}({\mathcal{P}_{{z_1}}})\\
 &\cong &p_{1*}(id, \beta)^*({\mathcal{P}_{{z_2}}})\\
 &\cong &\tilde{p}_{1^*}({\mathcal{P}_{{z_2}}})\\
 &\cong &F(-y)
 \end{array}
 $$
imply that $x=y$ and $E=F,$ which proves the proposition.

\end{proof}

\section{The Hilbert scheme}

 To compute the differential map $d\Upsilon $ at $z\in \mathcal{A}$ we denote by
 $N_{{\mathbb{G}/M}}$ the normal bundle of $\mathbb{G}(z)$
 in $M(n,\xi)$ and by $TM$ the tangent bundle of $M(n,\xi)$.
 These bundles fit into the following exact sequence
 \begin{equation}\label{eqtannorm}
 0\rightarrow T\mathbb{G}(z)\rightarrow \phi ^{{*}}_zTM \rightarrow N_{{\mathbb{G}/M}}\rightarrow 0
 \end{equation}
of vector bundles over $\mathbb{G}(z)$.
 Then Theorem \ref{teoprin} follows from the next proposition.

 \begin{proposition}\label{propprincipal} For any $z\in \mathcal{A}$,
 \begin{enumerate}
\item $H^0(\mathbb{G}(z),N_{{\mathbb{G}/M}})=T_{z}\mathcal{A}.$
\item $H^i(\mathbb{G}(z),N_{{\mathbb{G}/M}})=0 \ \  \mbox{for} \ \
i\geq 1.$ \item  $N_{{\mathbb{G}/M}}$ is generated by global
sections.
\end{enumerate}
 \end{proposition}

The proof of Proposition \ref{propprincipal} is somewhat lengthy,
so we will split it into several lemmas.

\smallskip

From diagram $(\ref{diaguni})$ we have that
 for any  $z\in \mathcal{A}$, $\pi
^{-1}_1(z)=\mathbb{G}(z)$ and $\Phi _{|\mathbb{G}(z)}=\phi _z$.
Hence, also from $(\ref{diaguni})$ we have the following commutative
diagram

 \begin{equation}\label{diagtodo}
\begin{array}{ccccccccc}
&& && 0& &0&& \\
&& && \downarrow& &\downarrow&& \\
&&&&T_{{\Phi}_{|\mathbb{G}(z)}} &= &T_{{\Phi}_{|\mathbb{G}(z)}}&& \\
&&&& \downarrow& &\downarrow&& \\
0&\rightarrow& T\mathbb{G}(z)&\rightarrow &
T\mathbb{G}(\mathcal{U})_{|\mathbb{G}(z)}& \rightarrow &
N_{{\mathbb{G}/\mathbb{G}(\mathcal{U})}}&\rightarrow &0\\
&& \|&& \downarrow& &\downarrow&& \\
0&\rightarrow &T\mathbb{G}(z)&\rightarrow& \phi ^{{*}}_zTM
& \rightarrow &N_{{\mathbb{G}/M}}&\rightarrow &0,\\
&&&& \downarrow& &\downarrow&& \\
&& && 0& &0&&
\end{array}
\end{equation}
where $N_{{\mathbb{G}/\mathbb{G}(\mathcal{U})}}$ is the normal
bundle of $\mathbb{G}(z)$ in $\mathbb{G}(\mathcal{U})$.

 As  $\mathbb{G}(z)$ is a fibre of the Grassmannian
 bundle $\pi _1:\mathbb{G}( \mathcal{U})\to \mathcal{A}$,
 the normal bundle $N_{{\mathbb{G}/\mathbb{G}(\mathcal{U})}}$ is the trivial bundle
 $\mathcal{O}_{\mathbb{G}(z)}\times T_z\mathcal{A}$.
Hence, we have the exact sequence
\begin{equation}\label{eqnormal}
0\to T_{{\Phi}_{|\mathbb{G}(z)}} \to
\mathcal{O}_{\mathbb{G}(z)}\times T_z\mathcal{A} \to
N_{{\mathbb{G}/M}} \to 0
\end{equation}
and the cohomology sequence
\begin{equation}\label{eqnormalcoho}
0\to H^0(T_{{\Phi}_{|\mathbb{G}(z)}}) \to
H^0(\mathcal{O}_{\mathbb{G}(z)})\times T_z\mathcal{A} \to
H^0(N_{{\mathbb{G}/M}}) \to H^1(T_{{\Phi}_{|\mathbb{G}(z)}}) \to \dots
\end{equation}
In order to compute the cohomology of $T_{{\Phi}_{|\mathbb{G}(z)}}$,
we restrict the extension $(\ref{eqpoincare})$ to the divisor $\Gamma
_{\gamma}=\mathbb{G} (\mathcal{U}) .$ To streamline the notation,
we use the same notation for the restrictions.

The next proposition is slightly different from \cite[Lemma 2.1, Proposition 4.8 and 4.12]{ram1},
but the proof works verbatim (see also \cite{tyurin} and Remark \ref{remkernel}).

\begin{proposition}\label{propprin} The restriction of $(\ref{eqpoincare})$
to the divisor $\Gamma _{\gamma}$ induces the exact sequence
\begin{equation}\label{eqkernel1}
  0\to \mathcal{K}\to {\mathcal{P}}_\mathcal{U}{\to}  p_2^*\pi _1^*\mathcal{U}\stackrel{{\alpha} }{\to}
 Q_\mathcal{U} \to 0
 \end{equation}
of vector bundles over $\Gamma _{\gamma}=\mathbb{G}
(\mathcal{U})$, where $\mathcal{K}$ is a vector bundle of rank $r$
such that $\mathcal{K}\cong Q_\mathcal{U}$. Moreover, the sequence
$(\ref{eqkernel1})$ splits as
\begin{equation}\label{eqkernel2}
  0\to \mathcal{K}\to {\mathcal{P}_\mathcal{U}}{\to} S_\mathcal{U} \to 0
  \ \ \mbox{and} \ \ 0\to S_\mathcal{U}\to p_2^*\pi _1^*\mathcal{U}\stackrel{{\alpha} }{\to}
 Q_\mathcal{U} \to 0.
 \end{equation}
 \end{proposition}

\bigskip

 Note that the restriction of
  the sequences $(\ref{eqkernel2})$ to $((x,E),W\subset E_x)\in \mathbb{G}
(\mathcal{U})$
 are the sequences
  $(\ref{eqiota1})$ in Remark \ref{remkernel}.

\begin{lemma}\label{lema2} The Grassmannian bundle
$g:\mathbb{G}(r,{\mathcal{P}_\mathcal{U}})\to \mathbb{G}
(\mathcal{U})$ has a section $$\sigma _{\mathcal{K}}:\mathbb{G}
(\mathcal{U})\to \mathbb{G}(r,{\mathcal{P}_\mathcal{U}})$$such
that
\begin{enumerate}  \item
$\sigma _{\mathcal{K}}^*(T_g)_{{|\mathbb{G}(z)}}\cong  \Omega
^1\mathbb{G}(z).$ \item $\sigma _{\mathcal{K}}^*( T_{{{{\pi
_2}_{|\sigma _{\mathcal{K}}(\mathbb{G}(z))}}}})\cong
(T_{\Phi})_{|\mathbb{G}(z)}$, where $\pi _2 :
\mathbb{G}(r,{\mathcal{P}_\mathcal{U}})\to M(n,\xi)$ is defined as
$((x,F),Z\subset F_x) \mapsto F$. \item
$H^i(\mathbb{G}(z),(T_{\Phi})_{|\mathbb{G}(z)})\cong
H^i(\mathbb{G}(z),\sigma _{\mathcal{K}}^*( T_{{{{\pi _2}_{|\sigma
_{\mathcal{K}}(\mathbb{G}(z))}}}}))$ for $i \geq 0.$
\end{enumerate}
\end{lemma}

\begin{proof} The subbundle $\mathcal{K}\subset \mathcal{P}_\mathcal{U} $ has rank $r$,
and defines the section $\sigma _{\mathcal{K}}:\mathbb{G}
(\mathcal{U})\to \mathbb{G}(r,{\mathcal{P}_\mathcal{U}})$ such
that $\sigma _{\mathcal{K}}^*(T_g)\cong \mathcal{K}^*\otimes
{\mathcal{P}_\mathcal{U}}/\mathcal{K}$. Thus, from Proposition
\ref{propprin} we have that $\sigma _{\mathcal{K}}^*(T_g) \cong
Q^*_\mathcal{U} \otimes S_\mathcal{U}\cong \Omega ^1_{{\pi_1}}$.

The rest of the lemma follows from the definition of $\sigma
_{\mathcal{K}}$, since it is clear that $\pi_2\circ \sigma
_{\mathcal{K}}= \Phi$,  and that $\sigma _{\mathcal{K}}$ is
isomorphic to its image. Therefore,
$\sigma _{\mathcal{K}}^*( T_{{{{\pi
_2}_{|\sigma _{\mathcal{K}}(\mathbb{G}(z))}}}})\cong T_{{\Phi}_{|\mathbb{G}(z)}}$ as claimed.

\end{proof}

The situation is summed up in, and we hope clarified by, the following commutative diagram

\begin{equation}\label{diagfin}
\xymatrix@1{ &
\mathbb{G}(\mathcal{U})\ar[ddl]_{\pi_{1}}\ar[ddr]_{\Phi}\ar@<1ex>[r]^{{\sigma
_{\mathcal{K}\ \ \ \ \ \ \ \ \ \ }}} &
\mathbb{G}(r,\mathcal{P_{U}})\subset
\mathbb{G}(r,\mathcal{P})\ar[dd]^{\pi_{2}}
\ar@<1ex>[l]^{g \ \ \ \ \ \ \ \  }\ar[rd]_{\pi} & \\
&&& X\times M(n,\xi),\ar[dl]_{p_{2}} \\
\mathcal{A}&&M(n,\xi)& }\\
\end{equation}
where $\pi:\mathbb{G}(r,{\mathcal{P}})\to X\times M(n,\xi)$ is the
Grassmannian Poincar\'{e} bundle.

\begin{remark}\begin{em} Strictly speaking, the morphisms in the above diagram are
defined in open sets. Let $\widetilde{\mathbb{G}\mathcal{U}}:=
\Phi  ^{-1}(\Phi(\mathbb{G}(\mathcal{U})))$ and
$\mathcal{A}^s:= \pi _1(\widetilde{\mathbb{G}\mathcal{U}}).$
 The open sets, and mainly the codimension of the complements, are relevant to
 compute cohomology groups of $M(n,\xi)$ (see \cite{ram1}).
However, in this paper, we will not use the cohomology groups of
$M(n,\xi)$ in any essential way, so we use
$\mathbb{G}(\mathcal{U}), \mathbb{G}(r,\mathcal{P_{U}}), M(n,\xi)$
as targets of our morphisms.
\end{em} \end{remark}

\begin{lemma}\label{lema4} $H^i(\sigma _{\mathcal{K}}(\mathbb{G}(z)),(T_{{\pi _2}})_{{|\sigma
 _{\mathcal{K}}(\mathbb{G}(z))}})=0$ for $i\geq 0.$
\end{lemma}
\begin{proof}
From Proposition \ref{prop1} we have that $\phi _z$ is an embedding. Therefore,
the image $\sigma _{\mathcal{K}}(\mathbb{G}(z)) $ is transversal
to the fibres of $\pi _2.$ Since the fibre of $\pi_2$ at  $E^W\in
\phi _z(\mathbb{G}(n-r, E_x))$ is the Grassmannian bundle
$$p:\mathbb{G}(r,E^W)\to X,$$ the restriction of  $(T_{{\pi _2}})$
to $\sigma _{\mathcal{K}}(\mathbb{G}(z))$
 fits into the following exact sequence
\begin{equation}\label{eqtangente}
  \zeta : 0\to (T_{g})_{{|_{\sigma _{\mathcal{K}}(\mathbb{G}(z))}}}\to
  (T_{{\pi _2}})_{{|\sigma _{\mathcal{K}}(\mathbb{G}(z))}}{\to}
  \mathcal{O}_{{\sigma _{\mathcal{K}}(\mathbb{G}(z))}}\times T_xX{\to} 0.
    \end{equation}

From the cohomology sequence of (\ref{eqtangente}),
$$H^i(\sigma
_{\mathcal{K}}(\mathbb{G}(z)), (T_{g})_{{|_{\sigma
_{\mathcal{K}}(\mathbb{G}(z))}}} )\cong H^i(\sigma
_{\mathcal{K}}(\mathbb{G}(z)),(T_{{\pi _2}})_{{|\sigma
 _{\mathcal{K}}(\mathbb{G}(z))}}) \ \  \mbox{for} \ \  i\geq 2$$
  since $H^i(\mathbb{G}(z),\mathcal{O}_{\mathbb{G}(z)})=0$ for $i\geq
  1.$ Moreover, for $i\geq 2,$ $$H^i(\sigma
_{\mathcal{K}}(\mathbb{G}(z)),(T_{{\pi _2}})_{{|\sigma
 _{\mathcal{K}}(\mathbb{G}(z))}})=0$$ by Lemma \ref{lema2},(1) and
 Remark $\ref{emgras0},(3)$.

The proof is completed by showing that
\begin{equation}\label{eqtangentecoho}
H^0(\sigma _{\mathcal{K}}(\mathbb{G}(z)),(T_{g})_{{|_{\sigma
_{\mathcal{K}}(\mathbb{G}(z))}}})=
 H^0(\sigma _{\mathcal{K}}(\mathbb{G}(z)),(T_{{\pi _2}})_{{|\sigma
 _{\mathcal{K}}(\mathbb{G}(z))}})=0.
 \end{equation}
The equality $(\ref{eqtangentecoho})$  follows from the cohomology
of the sequence (\ref{eqtangente}) and by recalling that $\sigma
_{\mathcal{K}}(\mathbb{G}(z))\subset
\mathbb{G}(r,{\mathcal{P}_x})$, $T_xX=\mathbb{C}$ and
$$H^1(\sigma _{\mathcal{K}}(\mathbb{G}(z)),(T_{g})_{{|_{\sigma
_{\mathcal{K}}(\mathbb{G}(z))}}})\cong H^1(\mathbb{G}(z), \Omega
^1{\mathbb{G}(z)})=\mathbb{C}.$$ It follows that $H^1(\sigma
_{\mathcal{K}}(\mathbb{G}(z)),(T_{{\pi _2}})_{{|\sigma
 _{\mathcal{K}}(\mathbb{G}(z))}})=0,$ and the proof is
 complete.

\end{proof}

{\it Proof of Proposition \ref{propprincipal}.} To compute the
cohomology of $N_{{\mathbb{G}/M}}$ we use the cohomology exact
sequence $(\ref{eqnormalcoho})$.

Combining Lemma \ref{lema2},(3) with Lemma \ref{lema4} we deduce
that $H^i(\mathbb{G}(z),T_{{\Phi}_{|\mathbb{G}(z)}})=0$ for $i\geq
0.$ We conclude from $(\ref{eqnormalcoho})$ that
$H^i(\mathbb{G}(z),N_{{\mathbb{G}/M}})=0 $ for $ i\geq 1,$ hence
that $H^0(\mathbb{G}(z),N_{{\mathbb{G}/M}})=T_{z}\mathcal{A},$
 and finally that
 $N_{{\mathbb{G}/M}}$ is generated by global sections, which
is the desired conclusion.

$ \hfill{\Box} $

The next proposition will be used to fix the Hilbert polynomial.

\begin{proposition}\label{propdegree} For any $z \in \mathcal{A}$, $\deg (\phi _z)=2n$.
\end{proposition}
\begin{proof} It follows from the exact sequence $(\ref{eqnormal})$ that
$$\deg (\phi _z^*(TM))= \deg (N_{{\mathbb{G}/M}}) + \deg
(T\mathbb{G}(z))= \deg (N_{{\mathbb{G}/M}}) +n.$$

We see at once that $\deg N_{{\mathbb{G}/M}}= n$, which is clear from the exact sequences in
 diagram $(\ref{diagtodo}).$ Therefore, $\deg (\phi _z^*(TM))=2n$ and
$\deg(\phi _z)=\deg(\phi _z ^*(-K_M))=2n$ as claimed.

\end{proof}

Applying  Proposition \ref{propdegree} we conclude that $\mathcal{HG}$ is a component of the Hilbert scheme
 $Hil^P_{M(n,\xi)}$ where $P$ is the Hilbert polynomial $$P(m)=\chi (\mathbb{G},\phi _z
^*(-K_M))=\chi (\mathbb{G},m(\mathcal{O}_{\mathbb{G}}(2n))).$$ We can now prove our main result, Theorem \ref{teoprin}

{\it Proof of Theorem \ref{teoprin}.}  From Proposition \ref{propdist} the morphism $\Upsilon :
\mathcal{A} \to \mathcal{HG},$ defined as $z\mapsto [\phi _z(\mathbb{G}(z))]$
is injective. Clearly, the
isomorphism $T_{z}\mathcal{A} \to
H^0(\mathbb{G}(z),N_{{\mathbb{G}/M}})$ in Proposition
\ref{propprincipal} is the differential of $\Upsilon $ at $z.$

Since $H^i(\mathbb{G}(z),N_{{\mathbb{G}/M}})=0 $ for $ i\geq 1$,
$\mathcal{HG}$ is smooth at $[\phi _z(\mathbb{G}(z))]$ for all
$z\in \mathcal{A}$ and $\dim\mathcal{HG}= (n^2-1)(g-1)+1$ (see Remark \ref{remtanggras}).
Moreover, from the exact sequence $(\ref{eqnormal})$,
$N_{{\mathbb{G}/M}}$ is  generated by global sections and, locally, the
deformations of an $r$-Hecke cycle are $r$-Hecke cycles.

$ \hfill{\Box} $

\section{The $Mor(\mathbb{G},M(n,\xi))$  scheme }

Let $\mathbb{G}$ be the Grassmannian
$\mathbb{G}(n-r,\mathbb{C}^n)$ and ${P}$  the Hilbert
polynomial $$P(m)=\chi
(\mathbb{G},m(\mathcal{O}_{\mathbb{G}}(2n))).$$  In this section we apply the
previous results to describe an open set of a component of the
scheme $Mor _{P}(\mathbb{G},M(n,\xi)).$

Let $Hom_{P}(\mathbb{G},M(n,\xi))$ be the scheme of morphisms from
$\mathbb{G}$ to $M(n,\xi)$. Recall that to remove the dependency
on the choice of coordinates of $\mathbb{G}$, we take the quotient
by the action of $Aut(\mathbb{G}).$ Therefore,
$$Mor_{P}(\mathbb{G},M(n,\xi))=
Hom_{P}(\mathbb{G},M(n,\xi))/Aut(\mathbb{G}).$$

\begin{proposition}\label{propdegree1}
Any embedding $\lambda:\mathbb{G}\to M(n,\xi)$ passing through general
points has degree at least  $2n,$ with respect to $-K_M$.
\end{proposition}
\begin{proof}
Suppose that $\lambda:\mathbb{G}\to M(n,\xi)$ has degree $t$.
 Let $\kappa : \mathbb{P}^1\to \mathbb{G}$ be a
morphism of degree $1$. Since $\lambda:\mathbb{G}\to M(n,\xi)$ is
an embedding, the composition $ \lambda \circ \kappa :
\mathbb{P}^1\to M(n,\xi)$ is a rational curve and, from
\cite[Theorem 1]{sun}, it has degree $\deg ( \lambda \circ
\kappa)=t\geq 2n,$ which is our claim.
\end{proof}

\begin{corollary}\label{cor22} Under the hypotheses of Proposition \ref{propkl}
 the $r$-Hecke morphisms have minimal degree.
\end{corollary}
\begin{proof} The conditions $(\ref{eqprinkl})$ imply that the vector
bundles in the image $\phi _z(\mathbb{G}(z))$ are
$(k,\ell)$-stable and hence very general (see Remark \ref{rem1r}).
Propositions \ref{propdegree1} and \ref{propdegree} now show that
$\deg (\phi_z) $ is minimal, which is the desired conclusion.
\end{proof}

The next proposition follows directly from Proposition
\ref{propdist}.

\begin{proposition}\label{propmorfismo} The morphism
$\Sigma :\mathcal{A}\to Mor_{P}(\mathbb{G},M(n,\xi)),$ defined as
$z\mapsto [\phi _z]$ is algebraic and injective. Moreover,
$\Sigma(\mathcal{A})$ is contained in an irreducible component of
$Mor_{P}(\mathbb{G},M(n,\xi))$.
\end{proposition}

Denote by $Mor_{P}^{\mathcal{H},r}(\mathbb{G},M(n,\xi))$ the
irreducible component of $Mor_{P}(\mathbb{G},M(n,\xi))$ that
contains the $r$-Hecke morphisms. We can now prove Theorem \ref{teoprin2}

\bigskip

{\it Proof of Theorem \ref{teoprin2}.} The theorem follows from
\cite{grot}. By the exact sequence $(\ref{eqtannorm})$ and
Proposition \ref{propprin},
$$\dim H^0(\mathbb{G}(z),\phi ^{{*}}_zTM)=
\dim H^0(\mathbb{G}(z), T\mathbb{G}(z))+ \dim
H^0(\mathbb{G}(z),N_{{\mathbb{G}/M}})$$ and $H^i(\mathbb{G}(z),
\phi _z^*(TM))=0$ for all $i\geq 1.$
Recall that the image of the
inclusion $H^0(\mathbb{G}(z), T\mathbb{G}(z)) \to
H^0(\mathbb{G}(z), \phi _z ^*(TM))$ corresponds to the
deformations of $\phi _z$ by reparameterisations. Hence, the
dimension of $Mor_{P}^{\mathcal{H},r}(\mathbb{G},M(n,\xi))$
 is $\dim H^0(\mathbb{G},N_{{\mathbb{G}/M}})=(n^2-1)(g-1)+1$. This
 completes the proof of the theorem.

$ \hfill{\Box} $

Let us mention two important consequences of the theorem.

\begin{corollary}\label{corprin} If $Mor_{2n}^{\mathcal{H},1}(\mathbb{P}^1,M(n,\xi))$
is the irreducible component of the space of Hecke curves in
$M(n,\xi)$ then
 $ \mathbb{G}(2,\mathcal{U}) \subseteq
 Mor_{2n}^{\mathcal{H},1}(\mathbb{P}^1,M(n,\xi)),$ where
  $ \mathbb{G}(2,\mathcal{U})$ is
the restriction of the Grassmannian Poincar\'e bundle to
$\mathcal{A}$.
\end{corollary}
\begin{proof} The proof is based on the following observation.
The Hecke curves are lines in the projective space
$\mathbb{P}(E_x)$. Therefore, the corollary follows from the
identification of $\mathbb{G}(1,\mathbb{P}(E_x))$ and
$\mathbb{G}(2, E_x)$.
\end{proof}

Let $Mor_s (\mathbb{P}^1,\mathbb{G})$ be the moduli space  of
stable maps from $\mathbb{P}^1$  to the Grassmannian $\mathbb{G}$,
of degree $s$.

\bigskip

 {\it Proof of Corollary  \ref{coroprin}.} The proof is
immediate from  the natural morphism
$$Mor _s(\mathbb{P}^1,\mathbb{G}) \times Mor _{P}(\mathbb{G}, M(n,\xi))
\to Mor _{2ns} (\mathbb{P}^1,M(n,\xi)),$$ induced by the
composition.

$ \hfill{\Box} $

\bigskip

The remainder of this section will be devoted to the proof of
Theorem \ref{teoprin3}.

\smallskip

Let us recall that if $L$ is an ample divisor on a projective
variety $Y$, the $L$-degree ${\mbox{deg}}_L (\mathcal{E})$ of a
torsion-free sheaf $\mathcal{E}$ on $Y$ is defined to be the
intersection number $[c_1(\mathcal{E})]\cdot [L]^{{\dim Y-1}}.$
The torsion-free sheaf $\mathcal{E}$ is said to be $L$-stable if,
for any proper subsheaf $\mathcal{F}$ of
$\mathcal{E}$,$$\frac{\deg _L\mathcal{F}}{n_ \mathcal{F}}<
\frac{\deg _L\mathcal{E}}{n_\mathcal{E}}.$$ Denote by
$M^L_{Y}(\mathcal{E})$  the moduli space of $L$-stable sheaves
over $Y$ with the same numerical invariants as $\mathcal{E}$.

For any $z \in \mathcal{A}$, the morphism $\phi _z$ is
defined by the existence of a vector bundle $\mathcal{P}_z$ over
$X\times \mathbb{G}$ called $r$-Hecke bundle. Since $Pic
(\mathbb{G})= \mathbb{Z}$, $Pic (X\times \mathbb{G})= Pic(X)
\oplus Pic(\mathbb{G}).$ Thus, any polarization $L$ on $X\times
\mathbb{G}$ can be expressed in the form $L =a\alpha +b\beta$ with
$a,b>0,$ where $\alpha$ is ample on $X$ and $\beta$
 ample on $\mathbb{G}$.

\begin{proposition}\label{teoprin311}  For any $z\in \mathcal{A},$ the $r$-Hecke
bundle $\mathcal{P}_z\to X\times \mathbb{G}$ is $L$-stable with
respect to any polarization $L$.
\end{proposition}

\begin{proof} The vector bundle $\mathcal{P}_z\to X\times \mathbb{G}$ is a family of
stable bundles parameterised by $\mathbb{G}.$ By construction, for
any $y\ne x$, $(\mathcal{P}_z) _{|_{{\{y\}\times
\mathbb{G}}}}\cong \mathcal{O}^n_{\mathbb{G}}$. Therefore, the
$L$-stability of $\mathcal{P}_z$ follows from \cite[Lemma
2.2]{bbn}.
\end{proof}

Let $M^{L,\mathcal{H}}_{{X\times \mathbb{G}}}(\mathcal{P}_z)$ be
the irreducible component of the  moduli spaces of $L$-stable
vector bundles over $X\times \mathbb{G}$ containing the $r$-Hecke
bundles $\mathcal{P}_z$, $z\in \mathcal{A}$.

\bigskip

{\it Proof of Theorem \ref{teoprin3}.}
The tangent space of the moduli space at $\mathcal{P}_z$ is
$$H^1(X\times \mathbb{G}, \mbox{End} (\mathcal{P}_z))=
H^1(X\times \mathbb{G},  \mathcal{O}_{X\times \mathbb{G}})
\oplus H^1(X\times \mathbb{G}, Ad(\mathcal{P}_z)).$$

From the Leray spectral sequence and the cohomology of
$(\ref{eqtannorm})$ we have that
$$ H^1(X\times \mathbb{G}, Ad(\mathcal{P}_z))=
H^0(\mathbb{G},\phi_z^*(TM) )=  H^0(\mathbb{G},
T\mathbb{G}(z))\oplus H^0(\mathbb{G},N_{{\mathbb{G}/M}}).$$ Since
$H^1(X\times \mathbb{G},  \mathcal{O}_{X\times \mathbb{G}}) =
H^1(X, \mathcal{O}_{X})$,
$$ H^1(X\times \mathbb{G}, End(\mathcal{P}_z))= H^0(\mathbb{G},
T\mathbb{G}(z))\oplus H^0(\mathbb{G},N_{{\mathbb{G}/M}})\oplus
H^1(X, \mathcal{O}_{X}).$$ Therefore, the theorem follows from the
equality $h^0(\mathbb{G},N_{{\mathbb{G}/M}})= \dim
T_z\mathcal{A}=(n^2-1)(g-1) +1$. Note that locally the
deformations of $\mathcal{P}_z$ come from those of the curve, from those of
the Grassmannian and from $H^0(\mathbb{G},\phi_z^*(TM) )$, and
this is precisely the assertion of the theorem.

$ \hfill{\Box} $

\bigskip

\section{Open questions}

There are a number of interesting questions about the $r$-Hecke
cycles, the Hecke morphisms and the Hecke bundles. We mention some
of them, which are most interesting from the viewpoint of moduli
spaces.

We have defined three open sets, namely:
\begin{enumerate}
\item $\Upsilon(\mathcal{A})$ of the Hilbert scheme
 $\mathcal{HG}$,
 \item  $\Sigma (\mathcal{A})$ of the moduli scheme
 $Mor_{P}^{\mathcal{H},r}(\mathbb{G},M(n,\xi)),$
 \item  $\Gamma (\mathcal{A})$
 of the moduli space $M^{L,\mathcal{H}}_{{X\times \mathbb{G}}}(\mathcal{P}_z)$.
\end{enumerate}

Let $\widetilde{\Upsilon}, \widetilde{\Sigma}$ and $\widetilde{\Gamma}$ be the closures of
  $\Upsilon(\mathcal{A})$, $\Sigma (\mathcal{A})$ and $\Gamma (\mathcal{A})$ respectively.

The natural questions are:

\begin{question} What are the elements that compactify these open sets?
Do the boundary points have a natural geometric meaning?
\end{question}

\begin{question} What are the relationships among
$\widetilde{\Upsilon}, \widetilde{\Sigma}$ and $\widetilde{\Gamma}?$
\end{question}

\begin{question} Do they give a new compactification of the moduli space $M(n,\xi)?$
\end{question}

For each $\phi _z: \mathbb{G}\to M(n,\xi)$, the image of the
differential $d\phi _z$ defines a subspace $d\phi_z T\mathbb{G}$
of $T_{E^W}M.$ The dual of the quotient $T_{E^W}M/d\phi_z
T\mathbb{G}$ defines a subspace $\mathbb{GH}$ in the cotangent
space $$T^*_{E^W}M=H^0(X, EndE^W\otimes K_X),$$ and hence a
subspace of the Higgs bundles and of the spectral curves.  This
construction generalises that of \cite{hram}.

\begin{question} Describe the Higgs bundles and the spectral
curves defined by $\mathbb{GH}.$
\end{question}

\bigskip

{\small Acknowledgments: We would like to thank T. G\'omez for his
contributions to our understanding of the Hecke correspondence. We
thank the Isaac Newton Institute, Cambridge, where this work
started, for its very stimulating mathematical environment, and to
CONACYT Grant-128250 for the support. The authors also thank the
referee for a careful reading of the manuscript and the many
comments. The first author is a member of the research group VBAC
and thanks J. Heinloth, R. Ramadas and U. Bruzzo for very helpful
conversations as well as V. Mu\~noz and V. Balaji for their
comments and suggestions about the paper. She also thanks CMR,
Barcelona and ICTP, Trieste  for their hospitality as well as for
the support during completion of this work, and for the great
research environment. }


\end{document}